\documentclass[12pt]{article}
\usepackage{e-jc}

\usepackage{amsmath,amsfonts,amssymb}
\usepackage{xspace}
\usepackage{url}
\usepackage{array}
\usepackage{tikz}

\usepackage{amsthm,amsmath,amssymb}

\usepackage{graphicx}

\usepackage[colorlinks=true,citecolor=black,linkcolor=black,urlcolor=blue]{hyperref}

\newcommand{\srg}{\mathrm{srg}}
\newcommand{\PG}{\mathrm{PG}}
\newcommand{\GQ}{\mathrm{GQ}}
\newcommand{\ID}{\gamma^{\rm{ID}}}
\newcommand{\IDf}{\gamma^{\rm{ID}}_f}
\newcommand{\LD}{\gamma^{\rm{LD}}}

\newcommand{\C}{\mathcal{C}} 

\theoremstyle{plain}
\newtheorem{theorem}{Theorem}
\newtheorem{lemma}[theorem]{Lemma}
\newtheorem{corollary}[theorem]{Corollary}
\newtheorem{proposition}[theorem]{Proposition}

\theoremstyle{definition}

\newtheorem{example}[theorem]{Example}

\theoremstyle{remark}
\newtheorem{remark}[theorem]{Remark}

\title{Identifying codes in vertex-transitive graphs and strongly regular graphs}

\author{Sylvain Gravier\\
\small Institut Fourier\\[-0.8ex] 
\small University of Grenoble\\[-0.8ex] 
\small Grenoble, France\\
\small\tt sylvain.gravier@ujf-grenoble.fr\\
\and
Aline Parreau \thanks{This work has been done when this author was an FNRS post-doctoral fellow at the University of Liege.}\\
\small LIRIS\\[-0.8ex]
\small University of Lyon, CNRS\\[-0.8ex] 
\small Lyon, France\\
\small\tt aline.parreau@univ-lyon1.fr\\
\and
Sara Rottey	\\
\small  Ghent University / Vrije Universiteit Brussel\\[-0.8ex]   
\small Ghent, Belgium / Brussels, Belgium\\
\small\tt sara.rottey@ugent.be\\
\and
Leo Storme\\
\small Ghent University\\[-0.8ex]   
\small Ghent, Belgium\\
\small\tt ls@cage.ugent.be\\
\and
\'Elise Vandomme	\\
\small Department of Mathematics / Institut Fourier\hspace{1cm}\\[-0.8ex]
\small  University of Liege / University of Grenoble\\[-0.8ex] 
\small Liege, Belgium / Grenoble, France\\
\small\tt e.vandomme@ulg.ac.be\\
}

\date{\dateline{May 12, 2015}{Sep 26, 2015} 
}

\begin{document}

\maketitle

\begin{abstract}
We consider the problem of computing identifying codes of graphs and its fractional relaxation.
The ratio between the size of optimal integer and fractional solutions is between 1 and $2\ln(|V|)+1$ where $V$ is the set of vertices of the graph. We focus on vertex-transitive graphs for which we can compute the exact fractional solution. There are known examples of vertex-transitive graphs that reach both bounds. We exhibit infinite families of vertex-transitive graphs with integer and fractional identifying codes of order $|V|^{\alpha}$ with $\alpha \in \{\frac{1}{4},\frac{1}{3},\frac{2}{5}\}$. These families are generalized quadrangles (strongly regular graphs based on finite geometries). They also provide examples for metric dimension of graphs.

\end{abstract}

\section*{Introduction}

Given a discrete structure on a set of elements, a natural question is to be able to locate efficiently the elements using the structure. If the elements are the vertices of a graph, one can use the neighbourhoods of the elements to locate them. In this context, Karpovsky, Chakrabarty and Levitin~\cite{KCL} have introduced the notion of identifying codes in 1998. An identifying code of a graph is a dominating set having the property that any two vertices of the graph have distinct neighbourhoods within the identifying code. Hence any vertex of the graph is specified by its neighbourhood in the identifying code.
Initially, identifying codes have been introduced to model fault-diagnosis in multiprocessor systems but later other applications were discovered such as the design of emergency sensor networks in facilities~\cite{UTS04}. They are related to other concepts in graphs like locating-dominating sets~\cite{S87,S88} and resolving sets~\cite{B80,S75}.

The problem of computing an identifying code of minimal size is NP-complete in general~\cite{CHL03} but can be naturally expressed as an integer linear problem. Also, one can ask how good the fractional relaxation of this problem can be. We focus on vertex-transitive graphs since for these graphs, we are able to compute the optimal size of a fractional identifying code. This value  depends only on three parameters of the graph: the number and degree of vertices and the smallest size of the symmetric difference of two distinct closed neighbourhoods. Moreover, the optimal cardinality of an integer identifying code is at most at a logarithmic factor (in the number of vertices $|V|$) of the fractional optimal value.

Identifying codes have already been studied in different classes of vertex-transitive graphs, especially in cycles~\cite{BCHL04,GMS06,JL12,XTH08} and hypercubes~\cite{BHL00,EJLR08,ELR08,HL02,KCL}.
In these examples, the order of the size of an optimal identifying code seems to always match its fractional value.
However, the smallest size of symmetric differences of closed neighbourhoods is small compared to the number of vertices: either it is constant (for cycles) or it has logarithmic order in the number of vertices (for hypercubes).
Therefore we focus in this paper on vertex-transitive strongly regular graphs that are vertex-transitive graphs with the property that two adjacent (respectively non-adjacent) vertices always have the same number of common neighbours. In particular, the size of symmetric differences can only take two values and is of order at least $\sqrt{|V|}$ if the graph is not a trivial strongly regular graph.

Another interest of considering identifying codes in strongly regular graphs is that they are strongly related to resolving sets. A resolving set is a set $S$ of vertices such that each vertex is uniquely specified by its distances to $S$. The minimum size of a resolving set is called the metric dimension of the graph. If a graph has diameter $2$ -- as is the case for non-trivial strongly regular graphs -- then a resolving set is the same as an identifying code except that the vertices of the resolving set are not identified. A consequence is that the optimal size of identifying codes and the metric dimension have the same order in strongly regular graphs. Actually, Babai~\cite{B80} introduced a notion equivalent to resolving sets in order to improve the complexity of the isomorphism problem for strongly regular graphs. He established an upper bound of order $\sqrt{|V|}\log_2(|V|)$ on the metric dimension of strongly regular graphs~\cite{B80,B81}. He also gave a finer bound using the degree $k$ of vertices of order $|V|\log_2(|V|)/k$~\cite{B81}. Thanks to this last bound, any family of strongly regular graphs for which $|V|$ is linear with $k$ have logarithmic metric dimension. This is for example the case of Paley graphs that have also been studied by Fijav\v{z} and Mohar~\cite{FM04} who gave a finer bound. Bailey and Cameron~\cite{BC11} proved that the metric dimension of some Kneser and Johnson graphs has order $\sqrt{|V|}$. Values for small strongly regular graphs have been computed~\cite{B13b,KCCKK08}. Recently, Bailey~\cite{B13a} used resolving sets in strongly regular graphs to compute the metric dimension of some distance-regular graphs.

Paley graphs give an example of an infinite family of graphs for which the optimal value of fractional identifying code is constant but the integer value is logarithmic, and so the gap between the two is also logarithmic.
We consider another family of strongly regular graphs that have never been studied in the context of identifying codes nor resolving sets: the adjacency graphs of generalized quadrangles. These graphs are constructed using finite geometries. Constructing identifying codes can be seen as a way to break the inherent symmetry of these graphs. We give constructions of identifying codes with size of optimal order. This order is of the form  $|V|^{\alpha}$ with $\alpha\in \{\frac{1}{4},\frac{1}{3},\frac{2}{5}\}$ and corresponds to the order of the fractional value.

\bigskip\noindent\textbf{Outline.} In Section~\ref{sec:bas}, we give formal definitions and classic results useful for the rest of the paper. In Section~\ref{sec:frac}, we exhibit the linear program for identifying codes, compute the optimal value  of the relaxation for vertex-transitive graphs and deduce a general bound. In Section~\ref{sec:previous}, we review known results for identifying codes in vertex-transitive graphs and compare them to our general bound. Finally in Section~\ref{sec:srg}, we study strongly regular graphs and in particular adjacency graphs of generalized quadrangles.

\section{Preliminaries}\label{sec:bas}

All the considered graphs are undirected, finite and simple. Let $G=(V,E)$ be a graph.
Let $u$ be a vertex of $G$. We denote by $N(u)$ the open neighbourhood of $u$, that is the set of vertices that are adjacent to $u$. We denote by $N[u]=N(u)\cup\{u\}$ the closed neighbourhood of $u$: $u$ and all its neighbours. The {\em degree} of a vertex is the number of its neighbours. A graph is {\em regular} if all vertices have the same degree. Given two vertices $u$ and $v$, we denote by $d(u,v)$ the distance between $u$ and $v$ that is the number of edges in a shortest path between $u$ and $v$. The {\em diameter} of $G$ is the maximum distance between any pair of vertices of the graph.
An {\em isomorphism} $\varphi: G=(V,E) \to G'=(V',E')$ between two graphs $G$ and $G'$ is a bijective application from $V$ to $V'$ that preserves the edges of the graph:  $uv$ is an edge of $G$ if and only if $\varphi(u)\varphi(v)$ is an edge of $G'$. If $G=G'$, $\varphi$ is called an {\em automorphism} of $G$. A graph is {\em vertex-transitive} if for any pair of vertices $u$ and $v$ there exists an automorphism sending $u$ to $v$. A vertex-transitive graph is in particular regular.

A subset of vertices $S$ is a {\em dominating set} if each vertex is either in $S$ or adjacent to a vertex in $S$. In other words, for every vertex $u$, $S\cap N[u]$ is non-empty.
A vertex $c$ {\em separates} two vertices $u$ and $v$ if exactly one vertex among $u$ and $v$ is in the closed neighbourhood of $c$. In other words, $c\in N[u]\Delta N[v]$ where $\Delta$ denotes the symmetric difference of sets.
A subset of vertices $S$ is a {\em separating set} if it separates every pair of vertices of the graph.
A subset of vertices $C$ is an {\em identifying code} if it is both a dominating and separating set. In other words, the set $N[u]\cap C$ is non-empty and uniquely determines $u$.
There exists an identifying code in $G$ if and only if $G$ does not have two vertices $u$ and $v$ with $N[u]=N[v]$. We say that two such vertices $u$ and $v$ are {\em twin vertices} and we will only consider twin-free graphs. The size of a minimal identifying code of $G$ is denoted by $\ID(G)$. We have the following general bounds.

\begin{proposition}[Karpovsky, Chakrabarty and Levitin~\cite{KCL}, Gravier and Moncel~\cite{GM07}]\label{prop:galbound}
Let $G$ be a twin-free graph with at least one edge. We have
$$\log_2(|V|+1)\leq \ID(G) \leq |V|-1.$$
\end{proposition}

The lower bound can be found by considering that in an identifying code $C$ of size $\ID(G)$, the sets $N[u]\cap C$ are all distinct and non-empty subsets  of a set of size $\ID(G)$. Both bounds are tight and graphs reaching the lower bound are described in~\cite{M06} whereas graphs reaching the upper bound are characterized in~\cite{FGKNPV10}.

When the maximum degree of the graph is small enough, the following lower bound is better than the previous one.

\begin{proposition}[Karpovsky, Chakrabarty and Levitin~\cite{KCL}]
Let $G$ be a graph of maximum degree $k$. We have
$$\ID(G)\geq \frac{2|V|}{k+1}.$$
\end{proposition}

Karpovsky \textit{et al.\ }prove this bound using a discharging method. We use the same method to obtain a tighter bound that we will need when $\ID(G)$ is smaller than the maximum degree of the graph.

\begin{proposition}\label{prop:lower}
Let $G=(V,E)$ be a twin-free graph of maximum degree $k$ and $C$ an identifying code of $G$ with $k\geq |C|+1$. We have
  $$|V|\leq \frac{|C|^2}{6}+\frac{(2k+5)|C|}{6}.$$
 \end{proposition}

\begin{proof}
We use the same discharging method as Karpovsky \textit{et al.\ }in~\cite{KCL}.
Each vertex receives a charge $1$ at the beginning.
Then each vertex $v$ gives to the vertices in $N[v]\cap C$ the charge $\frac{1}{|N[v]\cap C|}$.
After this process, only vertices of $C$ have a positive charge and the total charge is still $|V|$.

Let $c\in C$. Let $V_{i}$ be the set of vertices of $N[c]$ with exactly $i$ neighbours in $C$. Necessarily $|V_{1}|\leq 1$ since vertices in $V_{1}$ have only $c$ in their neighbourhood. We have $|V_2|\leq |C|-1$. Indeed, a vertex of $V_2$ has $c$ in its neighbourhood and a unique additional vertex of the code. But all the additional code neighbours of elements of $V_2$ must be different, hence there are at most $|C|-1$ vertices in $V_2$. Finally, there are $k+1-|V_1|-|V_2|$ other vertices giving charge at most $1/3$.
Therefore, $c$ receives a charge at most equal to $$|V_1|+\frac{|V_2|}{2}+\frac{k+1-|V_1|-|V_2|}{3}\leq 1+\frac{|C|-1}{2}+\frac{k-|C|+1}{3}=\frac{|C|}{6}+\frac{2k+5}{6}.$$ Hence the total charge $|V|$ is at most $\frac{|C|^2}{6}+\frac{(2k+5)|C|}{6}$.
\end{proof}

The concept of identifying codes is related to other concepts such as locating-dominat\-ing sets and resolving sets.
A {\em locating-dominating set} is a dominating set $S$ that separates the pairs of vertices that are not in $S$.
The size of a minimal locating-dominating set of $G$ is denoted by $\LD(G)$. Note that every graph admits a locating-dominating set since the whole set of vertices is one. An identifying code is always a locating-dominating set and one can obtain an identifying code from a locating-dominating set by adding at most $\LD(G)$ vertices. Therefore we have the following relations between $\LD(G)$ and $\ID(G)$.

\begin{proposition}[Gravier, Klasing and Moncel~\cite{GKM08}]\label{prop:IDLD}
Let $G$ be a twin-free graph. We have $$\LD(G)\leq \ID(G) \leq 2\LD(G).$$
\end{proposition}

A {\em resolving set} is a subset of vertices $S$ such that for every pair of vertices $u$ and $v$, there exists a vertex $x$ in $S$ that satisfies $d(x,u)\neq d(x,v)$. The smallest size of a resolving set of $G$ is called the {\em metric dimension} and is denoted by $\beta(G)$. A locating-dominating set is always a resolving set and so $\beta(G)\leq\LD(G)$. When the diameter of the graph is $2$, the reverse is almost true: adding a vertex to a resolving set gives a locating dominating set.

\begin{proposition}\label{prop:LDRS}
Let $G$ be a graph of diameter $2$. We have
$$\beta(G)\leq \LD(G) \leq \beta(G)+1.$$
\end{proposition}

\begin{proof}
The first part is true for any graph since a locating-dominating set is a resolving set.
Now let $S$ be a resolving set of a graph $G$ of diameter $2$. Order the vertices of $S=\{x_1,...,x_s\}$.
For every vertex $u$, let $L(u)=(d(u,x_1),...,d(u,x_s))$ be the distance vector to vertices of $S$.
Since $S$ is a resolving set, all the vectors $L(u)$ are distinct. Since the diameter is $2$,  $L(u)\in \{0,1,2\}^s$. But at most one vertex $u_0$ can have $L(u_0)=(2,2,...,2)$, hence all vertices except $u_0$ are dominated by a vertex of $S$. Therefore, the set $S'=S\cup\{u_0\}$ is a dominating set. Let $u$ be a vertex not in $S'$. It has only values $1$ and $2$ in its vector $L(u)$ and the set $N[u]\cap S$ is given by the value $1$ in $L(u)$. Hence all the sets $N[u]\cap S$ for $u\notin S'$ are distinct. Therefore, all the sets $N[u]\cap S'$ are also distinct for $u\notin S'$ and $S'$ is a locating-dominating set. In particular, $\LD(G) \leq \beta(G)+1$.
\end{proof}

Proposition~\ref{prop:IDLD} together with Proposition~\ref{prop:LDRS} gives a relation between $\ID(G)$ and the metric dimension in graphs of diameter $2$. In particular, they have the same order and let us derive results for identifying codes from results for resolving sets.

\begin{corollary}\label{cor:MDtoIC}
Let $G$ be a twin-free graph of diameter $2$. We have
$$\beta(G)\leq \ID(G) \leq 2\beta(G)+2.$$
\end{corollary}

\section{Fractional relaxation}\label{sec:frac}

The problem of finding a minimal identifying code in a graph $G$ can be expressed as a hitting set problem. Indeed an identifying code is a subset of $V$ that intersects all the sets $N[u]$ and $N[u]\Delta N[v]$ for $u,v\in V$. In other words, the problem of finding a minimal identifying code is equivalent to the following linear integer program $P_G$.

\renewcommand{\arraystretch}{2}
\begin{math}
\begin{array}{rclrc}
   \text{Minimize } & \displaystyle\sum_{x_u \in V} x_u &&&\\
   \text{such that } & \displaystyle \sum_{w\in N[u]} x_w &\geq 1 &  \forall u \in V& \text{(domination)} \\
   & \displaystyle \sum_{w\in N[u]\Delta N[v]} x_w &\geq 1 &  \forall u,v \in V, u\neq v & \text{(separation)} \\
    &x_u \in \{0,1\}&& \forall u \in V\\
\end{array}
\end{math}

Let us denote by $P^*_G$ the linear programming fractional relaxation of $P_G$ where the integrality condition $x_u \in \{0,1\}$ is replaced by a linear constraint $0\leq x_u\leq 1$ for all vertices $u\in V$.
The optimal value of $P^*_G$, denoted by $\IDf(G)$, gives an estimation on $\ID(G)$ within a logarithmic factor.

\begin{proposition}\label{prop:fracbound}
Let $G$ be a twin-free graph with at least three vertices. We have
$$\IDf(G)\leq \ID(G) \leq \IDf(G)(1+2\ln{|V|}).$$
\end{proposition}

\begin{proof}
The first inequality is trivial since $P^*_G$ is a relaxation of $P_G$.
Let $\mathcal H$ be the hypergraph with vertex set $V$ and hyperedge set $$\mathcal E=\{N[u]\ | \ u\in V\}\cup \{N[u]\Delta N[v] \ | \ u\neq v \in V\}.$$
The identifying code problem in $G$ is equivalent to the covering problem in $\mathcal H$ that is the problem of finding a set of vertices of minimum size that intersects all the hyperedges. The linear programming formulations are the same. Using the result of Lov\'asz~\cite{L75} on the ratio of optimal integral and fractional covers, we have
$$\ID(G) \leq \IDf(G)(1+\ln{r})$$
where $r$ is the maximal degree of $\mathcal H$, i.e.\ the maximum number of hyperedges a vertex is belonging to.
Let $u\in V$ and $k$ its degree in $G$. Then $u$ is in $k+1$ hyperedges of the form $N[v]$ and in $(|V|-k-1)(k+1)$ hyperedges of the form $N[v]\Delta N[w]$. Indeed, we must have $v\in N[u]$ and $w\notin N[u]$. Hence the degree of $u$ in $\mathcal H$ is $(|V|-k)(k+1)$.
The maximal value of $(|V|-k)(k+1)$ with $0\leq k \leq |V|-1$ is obtained for $k=\frac{|V|-1}{2}$.
Therefore,  $r\leq \frac{(|V|+1)^2}{2} \leq |V|^2$ for $|V|\geq 3$ which leads to the upper bound of the proposition.
\end{proof}

In the case of vertex-transitive graphs, we can compute the exact value of $\IDf$.

\begin{proposition}\label{prop:idftransitive}
Let $G$ be a twin-free vertex-transitive graph. Let $k$ denote the degree of $G$ and let $d$ denote the smallest size of symmetric differences of closed neighbourhoods $N[u]\Delta N[v]$ among all the pairs of distinct vertices $u,v$. We have
$$\IDf(G)=\frac{|V|}{\min(k+1,d)}.$$
In particular
$$\frac{|V|}{\min(k+1,d)} \leq \ID(G) \leq \frac{|V|(1+2\ln{|V|})}{\min(k+1,d)}.$$
\end{proposition}

\begin{proof}
Giving to each variable $x_u$ the value $\frac{1}{\min(k+1,d)}$ leads to a feasible solution of $P^*_G$, hence
$$\IDf(G)\leq\frac{|V|}{\min(k+1,d)}.$$

Since $G$ is a vertex-transitive graph, all the vertices play the same role. Consider the finite set $\mathcal{S}$ of extreme optimal solutions (solutions that are vertices of the polytope defined by $P^*$). Any linear combination  of elements of $\mathcal{S}$, with the sum of coefficients equal to $1$ is still an optimal solution of $P^*$. In particular, ${\bf x}=\frac{1}{|S|}\sum_{s\in \mathcal S}s$ is an optimal solution. We claim that all the components of ${\bf x}$ are equal. Indeed, assume that $x_u\neq x_v$ and let $\varphi$ be an automorphism sending $u$ to $v$. Let $s\in \mathcal S$, then $\varphi(s)$ and $\varphi^{-1}(s)$, obtained by permuting the value inside $s$ following the automorphism $\varphi$ are still extreme optimal solutions. Hence $\mathcal S$ is stable by $\varphi$ and so $\varphi({\bf x})={\bf x}$, a contradiction since $x_u\neq x_v$.
\end{proof}

\section{Known results on vertex-transitive graphs}\label{sec:previous}

We review some known results on classes of transitive graphs. In particular, we discuss the gap between $\ID$ and $\IDf$. Sometimes, not only identifying codes but also {\em $r$-identifying codes} have been studied in these classes. Instead of using the closed neighbourhoods, that are the balls of radius $1$, one consider the balls of radius $r$ to identify the vertices. It is equivalent to consider $r$-identifying codes in a graph $G$ or to consider identifying codes in $G^r$, the $r^{\mathrm{th}}$-power of $G$, obtained by adding edges between each pair of vertices of $G$ that are at distance at most $r$. In the following, we will express the results in terms of identifying codes in the power graph.

\subsection{Cycles}

We first consider cycles and powers of cycles.
Let $n,r\in \mathbb N$ with $n\geq 5$ and $1\leq r<\frac{n-1}{2}$. The cycle on $n$ vertices, $\mathcal C_n$, has vertex set $V=\{0,1,...,n-1\}$ and two distinct vertices $i$ and $j$ are adjacent if $|i-j|=1$ (modulo $n$). The graph $\mathcal C_n^r$ is vertex-transitive with vertex degree $2r$. The smallest symmetric difference of closed neighbourhoods has size $2$. It is obtained via two consecutive vertices $i$ and $i+1$ whose symmetric difference of closed neighbourhoods is the set $\{i-r,i+r+1\}$ (modulo $n$). Hence the optimal value of fractional identifying codes is $\IDf(\mathcal C_n^r)=\frac{n}{2}$.

On the other hand, the study of integer identifying codes in powers of cycles had taken several years (see e.g.~\cite{BCHL04,GMS06,XTH08}) before being completed by Junnila and Laihonen~\cite{JL12}. We have the following results.
If $n$ is even and at least $2r+4$, then $$\ID(\mathcal C_n^r)=\frac{n}{2}=\IDf(\mathcal C_n^r).$$

If $n$ is odd and at least $2r+3$, then
$$\frac{n+1}{2}\leq \ID(\mathcal C_n^r)\leq \frac{n+1}{2}+r.$$

In particular, the difference between $\ID(\mathcal C_n^r)$ and $\IDf(\mathcal C_n^r)$ is bounded by $r$. Hence the ratio is converging to $1$ when $r$ is fixed and $n$ is large.

When $n=2r+2$, $\mathcal C_n^r$ is a complete graph where a perfect matching is removed and we have $\ID(\mathcal C_n^r)=n-1$. Then $\frac{\ID(\mathcal C_n^r)}{\IDf(\mathcal C_n^r)} \to 2$ when $n$ is large. Finally, if $n=2r+3$, $\ID(\mathcal C_{n}^{r})=\lfloor \frac{2n}{3} \rfloor$ and the ratio is converging to $4/3$.

\subsection{Hypercubes}

Let $\ell\geq 3$. The vertex set of the hypercube of dimension $\ell$, denoted by $\mathcal H_\ell$, is the set of binary words of length $\ell$, $\{0,1\}^\ell$. Two vertices are adjacent if the corresponding words differ on exactly one letter.
Clearly, $\mathcal H_\ell$ is vertex-transitive with vertex degree $k=\ell$. The smallest symmetric difference of closed neighbourhoods has size $d=2\ell-2$ and is obtained via two adjacent vertices. Hence, by Proposition~\ref{prop:idftransitive}, $$\IDf(\mathcal H_\ell)=\frac{2^\ell}{\ell+1}.$$

Computing the exact value of $\ID(\mathcal H_\ell)$ seems difficult and only few exact values are known. However, we have the following bounds (see~\cite[Theorem 4]{EJLR08} for the upper bound and~\cite{KCL} for the lower bound)
\begin{equation*}\label{eq:r1}
\frac{\ell2^{\ell+1}}{\ell(\ell+1)+2}\leq \ID(\mathcal H_\ell)\leq \frac{9}{2}\cdot\frac{2^\ell}{\ell+1}.
\end{equation*}

Hence the optimal values of integer and fractional identifying codes have the same order and the ratio satisfies

$$2-\frac{4}{\ell(\ell+1)+2}\leq \frac{\ID(\mathcal H_\ell)}{\IDf(\mathcal H_\ell)}\leq \frac{9}{2}.$$

Let $1<r<\ell$. We now consider $r$-identifying codes or equivalently identifying codes in $\mathcal H_\ell^r$. The graph $\mathcal H_\ell^r$ is still vertex-transitive. The degree of the vertices is $k=\sum_{i=1}^r{\ell\choose i}$. The smallest symmetric difference of closed neighbourhoods has now size $d=2{\ell-1 \choose r}$ and is still obtained via two adjacent vertices of $\mathcal H_\ell$. Thus, by Proposition~\ref{prop:idftransitive},
 $$\IDf(\mathcal H_\ell^r)=\frac{2^{\ell}}{\min\left(\sum_{i=0}^r{\ell\choose i},2{\ell-1\choose r}\right)}.$$

Concerning the general behaviour of $\ID(\mathcal H_\ell^r)$, we consider two cases: $r$ is fixed or $r$ is linearly dependent on $\ell$. Assume first that $r$ is fixed and $\ell$ is large. The bounds given by Karpovsky \textit{et al.\ }\cite{KCL} can be translated as follows. There are two constants $\alpha$ and $\beta$ (depending on $r$) such that, for large $\ell$,

\begin{equation}\label{eq:withcst}
\alpha\cdot\frac{2^\ell}{\ell^r} \leq \ID(\mathcal H_\ell^r) \leq \beta \cdot \frac{2^\ell}{\ell^r}.
\end{equation}

Thus $\ID(\mathcal H_\ell^r)$ and $\IDf(\mathcal H_\ell^r)$ have the same order, that is $2^\ell/\ell^r$.

\

Assume now that $r=\lfloor \rho \ell\rfloor$ for some constant $\rho$. Honkala and Lobstein~\cite{HL02} proved that $$\lim_{\ell\to \infty} \frac{ \log_2{\ID(\mathcal H_\ell^r)}}{\ell}=1-h(\rho)$$
where $h(x)=-x\log_2(x)-(1-x)\log_2(1-x)$ is the binary entropy function.
This result can be proved with Proposition~\ref{prop:fracbound}. Indeed,  $\frac{\log_2{\sum_{i=0}^r{\ell\choose i}}}{\ell}$ and $\frac{\log_2{{\ell \choose r}}}{\ell}$ tend to $h(\rho)$. Hence
$$\lim_{\ell\to \infty} \frac{\log_2{\IDf(\mathcal H_\ell^{r})}}{\ell}=1-h(\rho)$$
and
$$\lim_{\ell\to \infty} \frac{\log_2\left(\IDf(\mathcal H_\ell^{r})(1+2\ln{2^\ell})\right)}{\ell}=1-h(\rho).$$
But we do not know if $\ID(\mathcal H_\ell^r)$ and $\IDf(\mathcal H_\ell^r)$ have the same order in this case.

\subsection{Product of graphs}\label{sec:product}
One can easily obtain other vertex-transitive graphs by forming products of vertex-transitive graphs such as cliques. This was already the case for the hypercube which is the Cartesian product of $\ell$ cliques of size~$2$.
Identifying codes in the following products of graphs have been recently considered. Let $G=(V_G,E_G)$ and $H=(V_H,E_H)$ be two graphs. All the products we are using have vertex set $V_G\times V_H$. We follow notation and terminology of \cite{graphproduct}.

\begin{itemize}
\item For the {\em Cartesian product} $G\square H$, two vertices $(u_G,u_H)$ and $(v_G,v_H)$ are adjacent if either $u_G=v_G$ and $u_Hv_H\in E_H$ or $u_H=v_H$ and $u_Gv_G\in E_G$.
\item For the {\em direct product} $G\times H$, two vertices $(u_G,u_H)$ and $(v_G,v_H)$ are adjacent if $u_Gv_G\in E_G$ and $u_Hv_H\in E_H$.
\item For the {\em lexicographic product} $G\circ H$, two vertices $(u_G,u_H)$ and $(v_G,v_H)$ are adjacent if either $u_Gv_G\in E_G$ or $u_G=v_G$ and $u_Hv_H\in E_H$.
\end{itemize}

 \bigskip\noindent \textbf{Cartesian product of two cliques.} Let $2\leq p\leq q$ be integers. The Cartesian product $K_p\square K_q$ of two cliques is a vertex-transitive graph with vertex degree $k=p+q-1$. The smallest symmetric difference of closed neighbourhoods has size $d=2p-2$ and is obtained via two adjacent vertices. By Proposition~\ref{prop:idftransitive}, $$\IDf(K_p\square K_q)=\frac{pq}{2p-2}.$$

Identifying codes in $K_p\square K_q$ have been studied by Gravier, Moncel and Semri~\cite{GMS08} and by Goddard and Wash~\cite{GW13}. They proved that
\[
\ID(K_p\square K_q)=
\begin{cases}
q+\lfloor \frac{p}{2} \rfloor & \text{if } q\leq \frac{3p}{2}\\
2q-p & \text{if } q\geq \frac{3p}{2} \\
\end{cases}
\]

Therefore, the ratio between the optimal values of integer and fractional identifying codes is

\[
\frac{\ID(K_p\square K_q)}{\IDf(K_p\square K_q)}=
\begin{cases}
2+\frac{p}{q}-\frac{2}{p}-\frac{1}{q} & \text{if } q\leq \frac{3p}{2}\\
4-\frac{2p}{q}-\frac{4}{p}+\frac{2}{q} & \text{if } q\geq \frac{3p}{2} \\
\end{cases}
\]

In particular, it is bounded by a constant. 

Note that the metric dimension of Cartesian product of graphs and in particular of cliques have been studied by C\'aceres {\em et al.}~\cite{CHMPPSW07}.

 \bigskip\noindent \textbf{Direct product of cliques.} Let $2\leq p \leq q$ be integers. The direct product $K_p\times K_q$ of two cliques is a vertex-transitive graph with vertex degree $k=(p-1)(q-1)$. The smallest symmetric difference of closed neighbourhoods has size $d=2p$ and is obtained via two vertices belonging to the same copy of $K_q$. By Proposition~\ref{prop:idftransitive},
\[
\IDf(K_p\times K_q)=
\begin{cases}
\frac{q}{2} & \text{if $p\geq 4$ or $q>p$}\\
\frac{pq}{(p-1)^2+1} & \text{if $p\leq 3$ and $p=q$}.\\
\end{cases}
\]

Rall and Wash~\cite{RW14} gave the exact size of optimal identifying codes in $K_p\times K_q$. Except the small values of $p$ and $q$, there are two main cases. If $p\geq 3$ and $q\geq 2p$, then $\ID(K_p\times K_q)=q-1$. If $p\geq 5$ and $q<2p$, then $\ID(K_p\times K_q)$ is either $\lfloor \frac{2(p+q)}{3}\rfloor$ or $\lceil \frac{2(p+q)}{3}\rceil$ depending on the value of $p+q$ modulo $3$. Therefore, the ratio between the optimal values of integer and fractional identifying codes is either $2-2/q$ or $4/3(1+p/q)$, and it is again bounded.

 \bigskip\noindent \textbf{Lexicographic product of graphs.}
Let $G$ and $H$ be two vertex-transitive graphs that are not complete graphs. Then $G\circ H$ is also vertex-transitive. If $G$ (respectively $H$) has vertex degree $k_G$ (resp. $k_H$) and $n_G$ (resp. $n_H$) vertices, then $G\circ H$ has $n_Gn_H$ vertices and vertex degree $k=k_Gn_H+k_H$. Moreover, the size of the smallest symmetric difference of closed neighbourhoods of $G\circ H$ and $H$ are equal. Hence $$\IDf(G\circ H)=\frac{n_Gn_H}{d_H}$$ where $d_H$ is the smallest symmetric difference of closed neighbourhoods of $H$.

Assume that $G$ does not have two vertices $u$ and $v$ such that $N(u)=N(v)$. Feng \textit{et al.\ }\cite{FXW12} proved that in this case $$\ID(G\circ H)=n_Gs_H$$ where $s_H$ is the minimum size of a separating set of $H$. Hence we have
$$\frac{\ID(G\circ H)}{\IDf(G\circ H)}=\frac{s_Hd_H}{n_H}.$$
If $H$ is such that $k_H+1\geq d_H$, then $\IDf(H)=\frac{n_H}{d_H}$. Since $s_H$ is either equal to $\ID(H)$ or $\ID(H)-1$, the ratio between $\ID(G\circ H)$ and $\IDf(G\circ H)$ is the same as the ratio between $\ID(H)$ and $\IDf(H)$. In particular, if we have a ratio $\alpha$ for a graph $H$ we can obtain graphs with arbitrary sizes and still ratio $\alpha$.

\section{Strongly regular graphs}\label{sec:srg}
\subsection{General remarks}

The bound of Proposition \ref{prop:idftransitive} is helpful when the symmetric differences are large (larger than $\ln{|V|}$). For this reason, we now focus on the family of {\em strongly regular graphs} for which the smallest symmetric difference has, in most cases, size at least $\sqrt{|V|}$ (see Proposition~\ref{prop:bounddsrg}).

A {\em strongly regular graph}  $\srg(n,k,\lambda,\mu)$ is a $k$-regular graph $G$ on $n$ vertices for which any pair of adjacent (respectively non-adjacent) vertices have exactly $\lambda$ (resp. $\mu$) neighbours in common.

The four parameters are related in the following way
    \begin{equation}\label{eq:relationparam}
(n-k-1)\mu = k(k-\lambda-1).
\end{equation}

This relation can be proved by considering one particular vertex $u$ and the partition of $V\setminus \{u\}$ between the neighbours $A=N(u)$ and the non-neighbours $B=V\setminus N[u]$ of $u$. The number of edges between $A$ and $B$ is $(n-k-1)\mu = k(k-\lambda-1)$.

The complement of a strongly regular graph is still a  strongly regular graph and has parameters $\srg(n,n-1-k,n-2-2k+\mu,n-2k+\lambda).$ A strongly regular graph is {\em primitive} if the graph and its complement are connected.

\begin{example}\label{ex:trivialsrg}
Let $G$ be a $\srg(n,k,\lambda,\mu)$.
A trivial non primitive case is given by $\mu=0$. Indeed, if $\mu=0$, then it is the disjoint union of complete graphs on $k+1$ vertices. In particular, $\lambda=k-1$. 

Another non primitive case is given by $\mu=k$. Then $G$ is a complete multipartite graph.
Necessarily, all the parts have the same size, $n-k$. Note that the complement of $G$ corresponds to the first graph.
\end{example}

The two graphs in the previous example are the only non primitive graphs.
\begin{lemma}\label{lem:prim}
Let $G$ be a strongly regular graph. $G$ is primitive if and only if $\mu\notin \{0,k\}$.
In particular, all primitive strongly regular graphs have diameter $2$.
\end{lemma}

\begin{proof}
As explained in Example \ref{ex:trivialsrg}, if $\mu\in\{0,k\}$ then $G$ is not primitive.
If $\mu \neq 0$, then two non-adjacent vertices have at least one vertex in common. Hence the diameter of $G$ is two and in particular, $G$ is connected.
Assume now that $\mu\neq k$. By Equation \eqref{eq:relationparam}, the value of $\mu$ for the complement of $G$, $n-2k+\lambda$, is not $0$. As before, it means that the complement of $G$ has diameter $2$ and is connected.
\end{proof}

We now turn to results concerning identifying codes.
Strongly regular graphs have been used by Gravier \textit{et al.\ }\cite{GJLR14} to provide families of graphs for which all the subsets of a given size are identifying codes. However, they did not study optimal identifying codes.
As mentioned in the introduction, resolving sets and metric dimension have been studied in several contexts for strongly regular graphs. In particular, Babai~\cite{B80} gave an upper bound on the size of the symmetric differences of open neighbourhood in strongly regular graphs which leads to bounds on the metric dimension. Following his ideas, we prove similar results for identifying codes.

We first compute the smallest size $d$ of the symmetric differences of closed neighbourhoods using $\lambda$ and $\mu$ and then give a general upper bound on $d$.

\begin{proposition}\label{prop:symdif}
Let $G$ be a strongly regular graph $\srg(n,k,\lambda,\mu)$.
Let $u$ and $v$ be two vertices of $G$.
If $u$ is adjacent to $v$, then $|N[u]\Delta N[v]|=2(k-1)-2\lambda$. Otherwise, $|N[u]\Delta N[v]|=2(k+1)-2\mu$.

Hence, the smallest symmetric difference of closed neighbourhoods is
$$d=\min(2(k-\lambda-1),2(k-\mu+1))=2k-2\max(\lambda+1,\mu-1).$$
If $G$ is vertex-transitive\footnote{Actually, it seems that almost all the strongly regular graphs are not vertex-transitive, see for example~\cite{C03}. However, all the strongly regular graphs we are considering in this paper are in fact vertex-transitive.},  we have
$$\IDf(G)=\frac{n}{\min(k+1,2(k-\lambda-1),2(k-\mu+1))}.$$
\end{proposition}

\begin{proof}
Let $u$ and $v$ be two adjacent vertices. There are $k-\lambda$ neighbours of $u$ that are not neighbours of $v$. But $v$ is counted in these vertices. Hence $|N[u]\setminus N[v]|=k-1-\lambda$ and we obtain the results. The computation for the non-adjacent case is similar.
\end{proof}

\begin{proposition}\label{prop:bounddsrg}
Let $G$ be a primitive strongly regular graph $\srg(n,k,\lambda,\mu)$ on $n$ vertices, then $k\geq \sqrt{n-1}$ and the smallest symmetric difference satisfies $d> \sqrt{n}-3$.
\end{proposition}

\begin{proof}
Since $G$ is primitive, by Lemma \ref{lem:prim}, it has diameter $2$. Thus there are at most $1+k+k(k-1)$ vertices in $G$. Hence $n\leq 1+k^2$ and we obtain the upper bound on $k$.

To prove the second inequality, we use a result of Babai~\cite{B80}: for every pair of vertices $u,v$ of a primitive strongly regular graph $|N(u)\Delta N(v)|>\sqrt{n}-1$. If $u$ and $v$ are adjacent, $|N[u]\Delta N[v]|=|N(u)\Delta N(v)|-2$ whereas if $u$ and $v$ are non adjacent, $|N[u]\Delta N[v]|=|N(u)\Delta N(v)|+2$. Hence $d>\sqrt{n}-3$.
\end{proof}

Using these bounds together with Proposition \ref{prop:idftransitive}, we obtain the following general bound for strongly regular graphs when they are vertex-transitive.

\begin{corollary}
Let $G$ be a primitive strongly regular graph $\srg(n,k,\lambda,\mu)$. If $G$ is vertex-transitive, we have
$$\ID(G) \leq \frac{n(1+2\ln{n})}{\sqrt{n}-3}.$$
In particular $\ID(G)=O(\sqrt{n}\ln{n})$.
\end{corollary}

\subsection{Known results on particular families}

The only strongly regular graphs for which we know optimal identifying codes are Cartesian and direct products of two cliques of the same size that we already mentioned in the previous section. The Cartesian product $K_p\square K_p$ is a strongly regular graph $\srg(p^2,2p-2,p-2,2)$ whereas $K_p\times K_p$ (that is the complement of $K_p\square K_p$) is a $\srg(p^2,(p-1)^2,(p-2)^2,(p-2)(p-1))$. We obtain results for some other families by considering previous work on metric dimension.

 \bigskip\noindent \textbf{Kneser and Johnson graphs (of diameter 2).}
Let $1\leq p \leq m$. The {\em Johnson graph $J(m,p)$} is the graph whose vertices are the subsets of size $p$ of a set of $m$ elements and two vertices are adjacent if the corresponding sets intersect in exactly $p-1$ elements.
Since the diameter of $J(m,p)$ is $\min(p,m-p)$, the graph $J(m,p)$ is a primitive strongly regular graph if and only if $p=2$ or $p=m-2$. Note that the two corresponding graphs are isomorphic and have parameters $\srg({m \choose 2}, 2(m-2),m-2,4)$.

The {\em Kneser graph $K(m,p)$} is the graph whose vertices are the subsets of size $p$ of a set of $m$ elements and two vertices are adjacent if the corresponding sets do not intersect. The Kneser graph $K(5,2)$ corresponds to the well known Petersen graph. The graph $K(m,p)$ is a primitive strongly regular graph if and only if $p=2$. The graph $K(m,2)$ is a $\srg({m\choose 2},{m-2\choose 2},{m-4 \choose 2},{m-3 \choose 2})$.

Bailey and Cameron~\cite{BC11} have computed the exact value of the metric dimension in $J(m,2)$ and $K(m,2)$.

\begin{proposition}[{\cite[Corollary 3.33]{BC11}}]
For $m\geq 6$, the metric dimension of the Johnson graph $J(m,2)$ and the Kneser graph $K(m,2)$ is $\frac{2}{3}(m-i)+i$ where $m\equiv i\pmod{3}$.
\end{proposition}

Using Corollary~\ref{cor:MDtoIC} we obtain a bound for identifying codes.

\begin{corollary}
Let $G$ be $K(m,2)$ or $J(m,2)$. We have
 $$\frac{2m}{3}\leq \ID(G) \leq \frac{4(m+1)}{3}.$$
In particular, $\ID(G)=\Theta(\sqrt{|V|})$.
\end{corollary}

To compute the fractional identifying code number, one just has to compute the value of the smallest symmetric difference using Proposition~\ref{prop:symdif}. For $K(m,2)$ and $m\geq6$, $\mu-1\geq \lambda +1$ and $2k-2\mu+2=2(m-1)\leq k+1$. Hence, for $m\geq 6$,  $$\IDf(K(m,2))=\frac{m(m-1)}{4(m-1)}=\frac{m}{4}.$$

For $J(m,2)$, $\lambda+1\geq \mu-1$ whenever $m\geq 4$ and $2k-2\lambda-2=2(m-3)=k$.
Hence $$\IDf(J(m,2))=\frac{m(m-1)}{4(m-3)}=\frac{m}{4}+2.$$
In all cases, we have $\IDf(G)=\Theta(\sqrt{|V|})$ and the fractional and integer values have the same order for these graphs.

 \bigskip\noindent \textbf{Paley graphs.}
The Paley graph $P_q$ is defined for a prime power $q\equiv 1\pmod 4$. Vertices are the elements of the finite field $\mathbb{F}_q$ on $q$ elements, and $a$ is adjacent to $b$ if $a-b$ is a square. They are strongly regular $\srg \left (q, \tfrac{1}{2}(q-1),\tfrac{1}{4}(q-5),\tfrac{1}{4}(q-1) \right )$.
Paley graphs have the particularity to have symmetric difference of closed neighbourhoods of order $|V|$, hence the fractional identifying code number is bounded by a constant and the identifying code number is of order $\log_2{|V|}$.

\begin{proposition}
Let $q$ be a prime power satisfying $q \equiv 1 \pmod 4$ and $q\geq 9$. We have $\IDf(P_q)=\frac{2q}{q-1}$ and thus
$$\log_2 (q+1) \leq \ID(P_q) \leq (2+o(1))(1+2\ln{q}).$$
In particular, $\ID(P_q)=\Theta(\log_2{|V|})$.
\end{proposition}

\begin{proof}
We first compute the value of $d$. We have
$$\max(\lambda+1,\mu-1)=\tfrac{q-1}{4}.$$
Thus $d=\frac{q-1}{2}<k+1=\frac{q+1}{2}$ and  $\IDf(P_q)=\frac{2q}{q-1}\leq 2+o(1)$.

The lower bound on $\ID(P_q)$ is the general lower bound of Proposition \ref{prop:galbound}.
For the upper bound, we use the bound of Proposition \ref{prop:fracbound} with $\IDf(P_q)$.
\end{proof}

Similar results were obtained for metric dimension.

\begin{proposition}[Fijav\v{z} and Mohar~\cite{FM04}]
Let $q$ be a prime power satisfying $q \equiv 1 \pmod 4$. Then the metric dimension of the Paley graph $P_q$ satisfies
$$\log_2{q} \leq \beta(P_q) \leq 2\log_2{q}.$$
In particular, $\beta(P_q)=\Theta(\log_2{|V|})$.
\end{proposition}

\subsection{Generalized quadrangles}

The graphs obtained from generalized quadrangles form another family of strongly regular graphs.  Let $s,t$ be positive integers. A {\em generalized quadrangle} $\GQ(s,t)$ is an incidence structure, i.e.\ a set of points and lines, such that
\begin{itemize}
\item there are $s+1$ points on each line,
\item there are $t+1$ lines passing through each point,
\item for a point $P$ that does not lie on a line $L$, there is exactly one line passing through $P$ and intersecting $L$.
\end{itemize}
Such an incidence structure has $(st+1)(s+1)$ points and $(st+1)(t+1)$ lines. A trivial example is the incidence structure given by a square grid of size $s\times s$ which is a $\GQ(s-1,1)$. The {\em dual} of a generalized quadrangle is obtained by reversing the role of the lines and the points. In particular, the dual of a $\GQ(s,t)$ is a $\GQ(t,s)$.

Adjacency graphs can be naturally obtained from generalized quadrangles: consider the points as vertices and two vertices are adjacent if the corresponding points belong to a common line. For example, the adjacency graph of the square grid is exactly the Cartesian product of two cliques of size $s$, $K_{s}\square K_{s}$, already mentioned in Section~\ref{sec:product}. By abuse of notation, $\GQ(s,t)$ will also denote the adjacency graph of a generalized quadrangle with parameters $s$ and $t$.

Observe that a $\GQ(s,t)$ is a strongly regular graph $\srg((st+1)(s+1),s(t+1),s-1,t+1)$. Indeed, any vertex has degree $k=s(t+1)$, any pair of adjacent vertices has $s-1$ common neighbours and any pair of non adjacent vertices has $t+1$ common neighbours. From these values, we can easily compute the smallest size of symmetric differences of closed neighbourhoods : $$d=2s(t+1)-2\max(s,t).$$
We have $d>k+1$ if and only if $\GQ(s,t)$ is not trivial, i.e.\ $s>1$ and $t>1$. In that case, the following inequalities, for which Cameron gave a short combinatorial proof~\cite{C74}, hold.

\begin{lemma}[Higman's inequality~\cite{H71,H74}] For a $\GQ(s,t)$, if $s>1$ and $t>1$, then $t\le s^2$ and dually $s\le t^2$.
\end{lemma}

From now on, we assume that $s>1$ and $t>1$. We obtain the following bounds on $\IDf$ for  generalized quadrangles.

\begin{proposition}\label{prop:fracboundsGQvt}
Let $G$ be a vertex-transitive $\GQ(s,t)$ with $s>1$ and $t>1$. Let $n$ denote the number of vertices of the graph $G$. We have
$$2^{-5/4}\cdot n^{1/4}\leq \IDf(G)\leq 2\cdot n^{2/5}.$$
\end{proposition}

\begin{proof}
Let $G$ be a vertex-transitive $\GQ(s,t)$ with $s>1$ and $t>1$. Then $n=(st+1)(s+1)$ is the number of vertices of $G$. We have by Proposition~\ref{prop:idftransitive}
$$\ID_f(G)=\frac{(st+1)(s+1)}{s(t+1)+1}=\frac{s^2t}{st+s+1}+1.$$
As $st<st+s+1<2st$, we obtain $\frac{1}{2}s<\ID_f(G)<2s$.

Moreover, using the previous lemma, we obtain
$$s^{5/2}\le s^2 t < s^2t+st+s+1=n\le s^4+s^3+s+1<2\cdot s^4.$$
So $\left(\frac{1}{2}n\right)^{1/4}<s<n^{2/5}$. It follows that $\left(\frac{1}{2}\right)^{5/4}n^{1/4}<\ID_f(G)<2n^{2/5}$.
\end{proof}

Constructions of $\GQ(s,t)$ are known only for $(s,t)$ or $(t,s)$ in the set
$$\{(q,q), (q,q^2),(q^2,q^3), (q-1,q+1)\}$$
where $q$ is a prime power. Many of them are based on finite geometries. Generalized quadrangles coming from finite classical polar spaces of rank $2$ are given in Table~\ref{tab:polar_spaces}. For more information on these geometric structures, see e.g. \cite{HT01}.
It is well known that these polar spaces give rise to generalized quadrangles and they are often referred to as the {\em classical generalized quadrangles} \cite{PT84}.
As an example, the Cartesian product $K_{q+1}\square K_{q+1}$ can be seen as the adjacency graph of the incidence structure $Q^+(3,q)$ obtained from the points of a hyperbolic quadric in a finite projective space (when $q$ is a prime power).

\begin{table}
\begin{center}
\setlength{\extrarowheight}{-3pt}
\begin{tabular}{|c|c|c||l|}\hline
	Polar space & Name & $(s,t)$ &\\\hline
 	$Q^+(3,q)$ & Hyperbolic&$(q,1)$ & a grid\\
 	$ Q(4,q)$ & Parabolic & $(q,q)$ & dual of $W(3,q)$\\
   	$Q^-(5,q)$ & Elliptic & $(q,q^2)$& dual of $H(3,q^2)$\\
    $H(3,q^2)$ & Hermitian & $(q^2,q)$ & dual of $Q^-(5,q)$\\
    $H(4,q^2)$& Hermitian & $(q^2,q^3)$ & \\
    $W(3,q)$ & Symplectic & $(q,q)$ & dual of $ Q(4,q)$\\\hline
\end{tabular}
\end{center}
\caption{The finite classical polar spaces of rank $2$.}\label{tab:polar_spaces}
\end{table}

There are other generalized quadrangles known, however they have the same parameters as the ones given in Table~\ref{tab:polar_spaces} or they have parameters $(q-1,q+1)$ or $(q+1,q-1)$. We provide identifying codes of optimal order for some cases. A summary of our results in generalized quadrangles is given in Table~\ref{tab:resultGQ}.

\begin{table}
\begin{center}
\setlength{\extrarowheight}{-3pt}
\begin{tabular}{|c|c|c|c|c|c|}\hline
  GQ & $(s,t)$ & $n$ & Lower bound & Upper bound & Order \\ \hline 	
$T_2^*(\mathcal O)$ & $(q-1,q+1)$ & $q^3$ & $3q-7$ & $3q-3$ & $n^{1/3}$\\ 
$Q(4,q)$ & $(q,q)$ & $(q^2+1)(q+1)$ & $5q-2$ & $3q-4$ & $n^{1/3}$\\ 
$Q^-(5,q)$ & $(q,q^2)$ & $(q^3+1)(q+1)$ & $5q$ & $3q+2$ & $n^{1/4}$\\ 
$H(3,q^2)$ & $(q^2,q)$ & $(q^3+1)(q^2+1)$ & $5q^2-2$ & $2q^2-2$ & $n^{2/5}$\\ \hline
\end{tabular}
\end{center}
\caption{Results obtained on optimal values of integer identifying codes of some generalized quadrangles.  The number of vertices is denoted by $n$. Every order matches the order of the optimal value of fractional identifying codes. Note that in the first line, $q$ must be a power of $2$ whereas in the other cases, $q$ is any prime power. Also in the first line, $\mathcal{O}$ is a hyperconic.}\label{tab:resultGQ}
\end{table}

\subsubsection{Identifying codes in $T_2^*(\mathcal O)$, a particular $\GQ(q-1,q+1)$}

\begin{proposition}\label{prop:T2*0}
Let $q>2$ be a power of $2$.
There exists a $\GQ(q-1,q+1)$ with an identifying code of size $3q-3=\Theta(n^{1/3})$ where $n$ is the number of vertices.
\end{proposition}

Before giving the proof, we will consider a particular construction of a $\GQ(q-1,q+1)$ and give some structural properties. This construction is done in finite projective geometry. We recall some definitions for readers unfamiliar with them. For more information on finite projective geometry, see e.g.~\cite{H71,HT01}.

Let $q$ be a power of $2$. We set ourselves in the $3$-dimensional projective space $\PG(3,q)$ over the finite field $\mathbb F_q$ of order $q$. The points of $\PG(3,q)$ can be described using four coordinates $(X_0,X_1,X_2,X_3) \in \mathbb F_q^4 \setminus\{\mathbf{0}\}$ where two coordinates that are proportional refer to the same point.
Consider the hyperplane $H_\infty$ of equation $X_0=0$ in $\PG(3,q)$ and the conic $\C$ of equation $X_1 X_3 - X_2^2=0$ in the hyperplane $H_\infty$. Any line of $H_\infty$ intersects $\C$ in 0, 1 or 2 points. A line intersecting $\C$ in one point is {\em tangent} to $\C$. There is a special point, $N(0,0,1,0)$, called the {\em nucleus} of $\C$, that lies on all tangents of $\C$. Then any other point of $H_\infty$ lies on exactly one tangent of $\C$. The set $\mathcal{O}=\C\cup\{N\}$ is a {\em hyperconic}. This set has the property that each line of $H_\infty$ intersects $\mathcal{O}$ in 0 or 2 points.

Now consider the following incidence structure $T_2^*(\mathcal{O})=(\mathcal{P},\mathcal{L})$, where the set $\mathcal{P}$ of points is the set of affine points, i.e.\ points of $\PG(3,q)$ not in $H_\infty$ and the set $\mathcal{L}$ of lines is the set of the lines through a point of $\mathcal{O}$ not lying in $H_\infty$. This incidence structure is well-known to be a $\GQ(q-1,q+1)$ (see for example \cite[Theorem 3.1.3]{PT84}).

We will now construct an identifying code in $T_2^*(\mathcal{O})$.
In $T_2^*(\mathcal{O})$, the neighbourhood of a point $P$ is composed of a cone $P\C$ (all the  lines going through $P$ and a point of $\C$) and the line $PN$, where the points of $H_\infty$ are removed. The common neighbours of two adjacent vertices are the $q-2$ points lying on the unique line incident with these two vertices. In the case of non adjacent vertices, we first determine the intersection of their two cones.

\begin{lemma}\label{lem:intersection-cones-T2}
Consider two distinct affine points $P$ and $Q$ such that $PQ \cap H_\infty \notin \mathcal{O}$. The intersection of the two cones $P\C$ and $Q\C$ consists of the points of the conic $\C$ and of points lying in a plane containing $N$ and $PQ\cap H_\infty$.
\end{lemma}

\begin{proof}
Consider two distinct affine points $P(1,a,b,c)$ and $Q(1,\alpha,\beta,\gamma)$ such that $PQ \cap H_\infty \notin \mathcal{O}$. Consider the cones $P\C$ and $Q\C$ in $\PG(3,q)$. It is clear that the conic $\C$ belongs to $P\C\cap Q\C$. Consider now a point $V(1,v_1,v_2,v_3)$ not lying in $H_\infty$. Then $V$ belongs to $P\C$ if and only if
\begin{eqnarray*}
  && (0,a-v_1, b-v_2,c-v_3)\in \C\\
 &\iff &(a-v_1)(c-v_3)-(b-v_2)^2=0 \\
 &\iff& (ac-b^2)- c v_1 - a v_3+ (v_1 v_3 - v_2^2)=0.
 \end{eqnarray*}
A similar computation holds for $V\in Q\C$. Hence $V\in P\C\cap Q\C$ implies that
\[(ac-b^2)-(\alpha\gamma-\beta^2) - (c-\gamma) v_1 - (a-\alpha) v_3=0.\]
So $V$ lies in the plane $\pi$ of equation $((ac-b^2)-(\alpha\gamma-\beta^2))X_0 - (c-\gamma) X_1 - (a-\alpha) X_3=0$.
Consider the intersection of $H_\infty$ and $\pi$. It is the line $\ell$ satisfying the equations $X_0=0$ and $- (c-\gamma) X_1 - (a-\alpha) X_3=0$. Clearly, the line $\ell$ contains the nucleus $N(0,0,1,0)$ and also the point $PQ\cap H_\infty=(0,a-\alpha,b-\beta,c-\gamma)$.
\end{proof}

\begin{remark} In the previous statement, the points, arising as the intersection of the two cones, lie in a plane containing $N$ and $PQ\cap H_\infty$, and they actually form a conic $\C'$. Since the two quadratic cones intersect in an algebraic curve of degree $4$ which already contains the conic $\C$, the remaining curve $\Gamma$ of degree $2$ is either a conic $\C'$, either a line with multiplicity $2$ or two lines. The last two cases are in contradiction with the fact that $P$ and $Q$ are two distinct points with $PQ \cap H_\infty \notin \C\cup\{N\}$. So it follows that $\Gamma=\C'$.
\end{remark}

\begin{corollary}
Consider two distinct non adjacent vertices $P$ and $Q$ of $T_2^*(\mathcal{O})$. Their common neighbours are $q$ points lying in a plane containing $N$ and $PQ\cap H_\infty$, a point of the line $PN$ and a point of the line $QN$.
\end{corollary}

\begin{proof}
Let $P$ and $Q$ be two distinct non adjacent vertices of $T_2^*(\mathcal{O})$. From the structure of the $\GQ(q-1,q+1)$, $P$ and $Q$ have $q+2$ common neighbours. Consider the lines $PN$ and $QN$. They intersect only in $N$. Since $P$ (respectively $Q$) has a unique projection $P'$ on $QN$ (resp. $Q'$ on $PN$), $P'$ and $Q'$ are two common neighbours. The $q$ other common neighbours come from the intersection of the two cones $P\C$ and $Q\C$. Hence, from the previous lemma, they lie on a plane containing $N$ and $PQ\cap H_\infty$.
\end{proof}

\begin{theorem}\label{thm:icT2O}
The affine points of three lines of $T_2^*(\mathcal{O})$ containing $N$ and spanning $\PG(3,q)$ form an identifying code of $T_2^*(\mathcal{O})$.
\end{theorem}

\begin{proof}
Consider three lines $\ell_1, \ell_2, \ell_3$ of $T_2^*(\mathcal{O})$ containing $N$ and spanning $\PG(3,q)$. The points of these lines form a dominating set since any point is either on one of these lines or has a unique projection on each line $\ell_i$. As each point has a unique projection on each line $\ell_i$, it is clear that two points on these lines are always separated. Similarly, a point incident with a line $\ell_i$ is always separated from a point not incident with $\ell_1, \ell_2$, or $\ell_3$.

Consider now two points $S_1$ and $S_2$ that do not lie on the lines $\ell_i$. Assume that these points are not separated. In other words, assume that $Q_1 \in \ell_1$, $Q_2 \in \ell_2$ and $Q_3 \in \ell_3$ are common neighbours of $S_1$ and $S_2$. If $S_1$ and $S_2$ are adjacent, then their common neighbours lie on the same line $S_1S_2$. Hence $Q_1, Q_2$ and $Q_3$ are collinear, a contradiction since $\ell_1, \ell_2, \ell_3$ span $\PG(3,q)$.

If $S_1$ and $S_2$ are not adjacent, then either $Q_1, Q_2, Q_3$ all lie in the plane, which is uniquely defined by the previous corollary, that contains the nucleus $N$ and $S_1 S_2 \cap H_\infty$, or at least one of them lies in the plane containing $S_1$, $S_2$ and $N$. In the first case, the three points $Q_1, Q_2, Q_3$ are all in the same plane containing $N$. Hence, $\ell_1, \ell_2, \ell_3$ are coplanar which is a contradiction. In the second case, suppose that $Q_1$ lies in the plane containing $S_1, S_2$ and $N$. It follows that $Q_1$ is incident with the line $S_1N$ or $S_2N$. It implies that either $S_1\in\ell_1$ or $S_2\in\ell_1$, which is a contradiction.

Therefore the set of points on $\ell_1, \ell_2, \ell_3$ is an identifying code of $T_2^*(\mathcal{O})$.
\end{proof}

\begin{proof}[Proof of Proposition~\ref{prop:T2*0}]
Let $\ell_1, \ell_2$ and $\ell_3$ be three lines incident with $N$ and spanning $\PG(3,q)$. Consider the set $C$ consisting of the points of $T_2^*(\mathcal{O})$ on $\ell_1, \ell_2, \ell_3$.  By Theorem~\ref{thm:icT2O}, this set is an identifying code of size $3q$. Let $Q_1$ be a point on $\ell_1$ and $Q_2, Q_3$ be its projections on respectively $\ell_2$ and $\ell_3$. The set $C\setminus \{Q_1, Q_2, Q_3\}$ is still a dominating set. Indeed, a point $P$ that does not lie on the lines $\ell_i$ can not have $Q_1, Q_2$ and $Q_3$ as neighbours. Otherwise, $Q_2$ would have two projections on the line $Q_1P$, namely $Q_1$ and $P$.

Moreover, we have a one-to-one correspondence between the sets $$(N[P]\cap C)\setminus \{Q_1, Q_2, Q_3\}\text{ and }N[P]\cap C$$ since we can easily determine which vertices are eventually missing in the first sets. Hence, $ C\setminus \{Q_1, Q_2, Q_3\}$ is an identifying code of $T_2^*(\mathcal{O})$ of size $3q-3$.
\end{proof}

The next proposition gives lower bounds on the size of an identifying code in any adjacency graph of a $\GQ(q-1,q+1)$. In particular, our previous construction is optimal for $q=4$ and close to a constant for the other cases.

\begin{proposition}\label{prop:lowT2O}
 Let $q$ be a power of $2$. Any identifying code of a $\GQ(q-1,q+1)$ has size at least $3q-7$. Moreover, it has size at least $9=3q-3$ if $q=4$, $19=3q-5$ if $q=8$, $42=3q-6$ if $q=16$ and $90=3q-6$ if $q=32$.
 \end{proposition}

 \begin{proof}
To prove this proposition, we use Proposition~\ref{prop:lower}. Any identifying code $C$ of a $\GQ(q-1,q+1)$, with $|C|<q^2+q-2$ satisfies the inequality
\[ q^3\le \frac{|C|^2}{6}+ \frac{(2(q^2+q-2)+5)|C|}{6}. \]
Hence, $|C|^2+(2q^2+2q+1)|C|-6q^3\ge0$. If there exists an identifying code of size $3q-8$, then the right-hand side of the inequality is equal to
\[ (3q-8)^2+(2q^2+2q+1)(3q-8)-6q^3=-q^2-61q+56\]
which is negative for all $q\ge 32$. This is a contradiction. Therefore, any identifying code of a $\GQ(q-1,q+1)$ has size at least $3q-7$. For small values of $q$, we can obtain a better bound using the same inequality. Since the expression $(3q-c)^2+(2q^2+2q+1)(3q-c)-6q^3$ is negative for $(q,c)\in\{(4,5),(8,6),(16,7),(32,7)\}$, any identifying code of a $\GQ(q-1,q+1)$ has size at least
\[\begin{cases}
8\phantom{0}=3q-4&\text{if }q=4\\
19=3q-5&\text{if }q=8\\
42=3q-6&\text{if }q=16\\
90=3q-6&\text{if }q=32.
\end{cases}\]

We can slightly improve the bound for $q=4$ using a technical analysis. The details can be found on the arXiv version of our paper \url{http://arxiv.org/pdf/1411.5275v1.pdf}.
\end{proof}

\subsubsection{Identifying codes in a parabolic quadric which is a $\GQ(q,q)$}

\begin{proposition}\label{prop:gqqq} Let $q$ be a prime power. There exists a $\GQ(q,q)$ with an identifying code of size $5q-2=\Theta(n^{1/3})$ where $n$ is the number of vertices.
\end{proposition}

 Before giving the proof, we will consider a particular construction of a $\GQ(q,q)$ and give some structural properties.

 Let $q$ be a prime power. Let $Q$ be the set of points of $\PG(4,q)$ that satisfy the equation $X_0^2+X_1X_2+X_3X_4=0$ ($Q$ is a parabolic quadric).

 \begin{lemma}[\cite{HT01,PT84}]\label{lem:consgqqq}
 The incidence structure $Q(4,q)$ obtained from the points of $Q$ and lines of $Q$ (i.e.\ lines of $\PG(4,q)$ included in $Q$) is a generalized quadrangle $\GQ(q,q)$. Moreover, the closed neighbourhood of a point $A$ of $Q(4,q)$ is exactly the intersection between a hyperplane $\pi_A$ (called the tangent hyperplane) and $Q$.
 \end{lemma}

 \begin{lemma}\label{lem:gqqqcoplanar}
 Let $A$ and $B$ be two non adjacent points of $Q$. The common neighbours of $A$ and $B$ are coplanar.
 \end{lemma}

 \begin{proof}
 Let $\pi_A$ (respectively $\pi_B$) be the hyperplane containing all the neighbours of $A$ (resp. $B$). Since $A$ and $B$ are non adjacent, $\pi_A$ and $\pi_B$ are two distinct hyperplanes (of dimension 3). The common neighbours of $A$ and $B$ are all located in the intersection of $\pi_A$ and $\pi_B$ which is a plane.
 \end{proof}

\begin{proof}[Proof of Proposition~\ref{prop:gqqq}]
We will construct an identifying code for $Q(4,q)$, which is, by Lemma \ref{lem:consgqqq}, a $\GQ(q,q)$.
Consider a hyperplane $\pi=\PG(3,q)$ intersecting $Q(4,q)$ in a hyperbolic quadric $Q^+(3,q)$ (for example the hyperplane $X_0=0$). The hyperbolic quadric is isomorphic to a grid $K_{q+1}\square K_{q+1}$.

Consider three lines $\ell_0$, $\ell_1$, $\ell_2$ of $Q^+(3,q)$ that are pairwise not intersecting. Consider two distinct points $P_1, P_2 \in \ell_2$ and take lines $M_1$ and $M_2$ through $P_1$ and $P_2$ respectively, both not contained in the $Q^+(3,q)$ and hence not lying in the $3$-space $\pi$.

The set of $3(q+1)+2q=5q+3$ points $\mathcal{S}=\ell_0 \cup \ell_1 \cup \ell_2 \cup M_1 \cup M_2$ is an identifying code. Since it contains a whole line, it is a dominating set.
A point $A$ on a line $N_1$ of $\mathcal{S}$ is clearly separated from all the points that are not on
$N_1$ since it is adjacent to all the points of $N_1$. The point $A$ is also separated from all the other points of $N_1$ since they have different projection on any line $N_2$ of $\mathcal{S}$ not intersecting $N_1$. Hence all the points of $\mathcal{S}$ are separated from all the other points.

Consider now a point of $Q^+(3,q)\setminus\mathcal{S}$. It has exactly three neighbours on $\ell_0, \ell_1, \ell_2$ (that are collinear). Two points of $Q^+(3,q)\setminus\mathcal{S}$ with the same projections on $\ell_0, \ell_1, \ell_2$ are necessarily collinear. Hence they have different neighbours on $M_1$ (if the projection on $\ell_2$ is not $P_1$) or on $M_2$ (otherwise).
Hence any point of $Q^+(3,q)\setminus\mathcal{S}$ has a unique set of neighbours.

A point $A$ that is not in $Q^{+}(3,q)$ has four or five neighbours in $\ell_0 \cup \ell_1 \cup \ell_2  \cup M_1 \cup M_2$. Since $A$ does not lie in $Q^{+}(3,q)$, the three points on $\ell_0$, $\ell_1$ and $\ell_2$ are not collinear, hence they span a plane, that is contained in $\pi$. The only points of $M_1$ and $M_2$ that could be contained in this plane are the intersection of $M_1$ and $M_2$ with $\pi$ which is exactly the points $P_1$ and $P_2$. Since $P_1$ and $P_2$ are both in $\ell_2$ they cannot be both in the neighbourhood of $A$. Finally, the neighbours of $A$ in $\mathcal{S}$ are not coplanar. Using Lemma~\ref{lem:gqqqcoplanar}, $A$ is separated from all the other vertices.

To conclude the proof, note that as before we can remove a point on each line of $\mathcal{S}$ and still have an identifying code (remove a point on $\ell_0$, which does not have $P_1$ or $P_2$ as a neighbour, and remove its 4 distinct projections on the other lines).
\end{proof}

The next proposition gives a lower bound on the size of any identifying code of a $\GQ(q,q)$. In particular, the order of our previous construction is optimal. The proof is similar to the proof of Proposition~\ref{prop:lowT2O}.

\begin{proposition}
Let $q$ be a prime power. Any identifying code of  a $\GQ(q,q)$ has size at least $3q-4$.
\end{proposition}

\subsubsection{Identifying codes in an elliptic quadric which is a $\GQ(q,q^2)$}

\begin{proposition}\label{prop:gqqq2}
Let $q$ be a prime power. There exists a $\GQ(q,q^2)$ with an identifying code of size $5q=\Theta(n^{1/4})$ where $n$ is the number of vertices.
\end{proposition}

Before giving the proof, we will consider a particular construction of a $\GQ(q,q^2)$ and give some structural properties.

 Let $q$ be a prime power. Let $Q$ be the set of points of $\PG(5,q)$ that satisfy the equation $f(X_0,X_1)+X_2X_3+X_4X_5=0$ where $f(X_0,X_1)=dX_0^2+X_0X_1+X_1^2$, $d \in \mathbb{F}_q$, is an irreducible binary quadratic form over $\mathbb{F}_q$ ($Q$ is an elliptic quadric).

 \begin{lemma}[\cite{HT01,PT84}]\label{lem:consgqqq2}
 The incidence structure $Q^-(5,q)$ obtained from the $(q^3+1)(q+1)$ points of $Q$ and the  $(q^3+1)(q^2+1)$ lines of $Q$ (i.e.\ lines of $\PG(5,q)$ included in $Q$) is a generalized quadrangle $\GQ(q,q^2)$. Moreover, the closed neighbourhood of a point $A$ of $Q^-(5,q)$ is exactly the intersection between a hyperplane $\pi_A$ (the tangent hyperplane of $A$) and $Q$.
 \end{lemma}

 \begin{lemma}\label{lem:gqqq23space}
 Let $A$ and $B$ be two non adjacent points of $Q$. The common neighbours of $A$ and $B$ lie in a 3-dimensional space.
 \end{lemma}

 \begin{proof}
 Let $\pi_A$ (respectively $\pi_B$) be the hyperplane containing all the neighbours of $A$ (resp. $B$). Since $A$ and $B$ are non adjacent, $\pi_A$ and $\pi_B$ are two distinct hyperplanes (of dimension 4). The common neighbours of $A$ and $B$ are all located in the intersection of $\pi_A$ and $\pi_B$ which is a 3-dimensional space.
 \end{proof}

\begin{proof}[Proof of Proposition~\ref{prop:gqqq2}]
We will construct an identifying code for $Q^{-}(5,q)$ which is a generalized quadrangle $\GQ(q,q^2)$.
Consider a line $\ell_0$ of $Q^{-}(5,q)$, take two distinct 3-spaces $\pi_1$ and $\pi_2$ of $\PG(5,q)$ intersecting each other only in $\ell_0$ such that $\pi_i\cap Q^{-}(5,q)=Q^{+}(3,q)$.
Take two lines $\ell_1, \ell_2$ in $\pi_1 \cap Q^{-}(5,q)$ such that $\ell_0$, $\ell_1$ and $\ell_2$ are pairwise non-intersecting. Using the geometry, one can always consider two lines $\ell_3, \ell_4$ in $\pi_2 \cap Q^{-}(5,q)$ such that $\ell_0$, $\ell_3$ and $\ell_4$ are pairwise non-intersecting.

We will prove that the set of $5(q+1)=5q+5$ points of $\mathcal{S}=\{\ell_i\}_{i=0,\ldots,4}$ is an identifying code.
Since $\mathcal{S}$ contains a whole line, the set $\mathcal{S}$ is a dominating set.

A point $A$ on a line $N_1$ of $\mathcal{S}$ is clearly separated from all the points that are not on
$N_1$ since it is adjacent to all the points of $N_1$. The point $A$ is also separated from all the other points of $N_1$ since they have different projections on any line $N_2$ of $\mathcal{S}$ not intersecting $N_1$. Hence all the points of $\mathcal{S}$ are separated from all the other points.

Any point of $(\pi_1 \cap Q^{-}(5,q)) \backslash \mathcal{S}$ has exactly three neighbours on $\ell_0, \ell_1, \ell_2$ (that are collinear). Moreover, two points of $(\pi_1 \cap Q^{-}(5,q)) \backslash \mathcal{S}$ with the same projection on $\ell_0, \ell_1, \ell_2$ are necessarily collinear. Hence, they have different neighbours on $\ell_3$. It follows that all the points of $(\pi_1 \cap Q^{-}(5,q)) \backslash \mathcal{S}$ are separated from all the other points. Equivalently, also all the points of $(\pi_2 \cap Q^{-}(5,q)) \backslash \mathcal{S}$ are separated from all the other points.

A point $P\in Q^{-}(5,q)$ not in $\pi_1 \cup \pi_2$ has five neighbours in $\mathcal{S}$. Since $P$ does not lie in $\pi_1$, the three points on $\ell_0$, $\ell_1$ and $\ell_2$ are not collinear, hence they span a plane of $\pi_1$, containing one point of $\ell_0$. Since $P$ does not lie in $\pi_2$, the three points on $\ell_0$, $\ell_3$ and $\ell_4$ are not collinear, hence they span a plane of $\pi_2$, containing one point of $\ell_0$. Now it is clear that the five neighbours of $P$ span a 4-space. Using Lemma \ref{lem:gqqq23space} it follows that the point $P$ is separated by $\mathcal{S}$ from all other points.

To conclude the proof, note that as before we can remove a point on each line of $\mathcal{S}$ and still have an identifying code (remove a point on $\ell_0$ and remove its 4 distinct projections on the lines $\ell_1, \ell_2, \ell_3, \ell_4$).
\end{proof}

The next proposition gives a lower bound on the size of any identifying code of a $\GQ(q,q^2)$. In particular, the order of our previous construction is optimal. The proof is similar to the proof of Proposition~\ref{prop:lowT2O}.

\begin{proposition}
Let $q$ be a prime power. Any identifying code of  a $\GQ(q,q^2)$ has size at least $3q+2$.
\end{proposition}

\subsubsection{Identifying codes in a Hermitian variety which is a $\GQ(q^2,q)$}

\begin{proposition}\label{prop:gqq2q}
Let $q$ be a prime power. There exists a $\GQ(q^2,q)$ with an identifying code of size $5q^2-2=\Theta(n^{2/5})$ where $n$ is the number of vertices.
\end{proposition}

Before giving the proof, we will consider a particular construction of a $\GQ(q^2,q)$ and give some structural properties.

 Let $q$ be a prime power. Let $H$ be the set of points of $\PG(3,q^2)$ that satisfy the equation $X_0^{q+1}+X_1^{q+1}+X_2^{q+1}+X_3^{q+1}=0$ ($H$ is a Hermitian variety).

 \begin{lemma}[\cite{HT01,PT84}]\label{lem:consgqq2q}
 The incidence structure $H(3,q^2)$ obtained from the $(q^3+1)(q^2+1)$ points of $H$ and the $(q^3+1)(q+1)$ lines of $H$ (i.e.\ lines of $\PG(3,q^2)$ included in $H$) is a generalized quadrangle $\GQ(q^2,q)$. Moreover, the closed neighbourhood of a point $A$ of $H(3,q^2)$ is exactly the intersection between a plane $\pi_A$ (the tangent hyperplane of $A$) and $H$.
 \end{lemma}

 It is well known that the dual of $H(3,q^2)$ is $Q^-(5,q)$, see \cite[3.2.3]{PT84}.

 \begin{lemma}\label{lem:gqq2q3space}
 Let $A$ and $B$ be two non adjacent points of $H$. The common neighbours of $A$ and $B$ lie on a line.
 \end{lemma}

 \begin{proof}
 Let $\pi_A$ (respectively $\pi_B$) be the hyperplane containing all the neighbours of $A$ (resp. $B$). Since $A$ and $B$ are non adjacent, $\pi_A$ and $\pi_B$ are two distinct planes. The common neighbours of $A$ and $B$ are all located in the intersection of $\pi_A$ and $\pi_B$ which is a line.
 \end{proof}

\begin{proof}[Proof of Proposition~\ref{prop:gqq2q}]
We will construct an identifying code for $H(3,q^2)$ which is a generalized quadrangle $\GQ(q^2,q)$.

Consider three disjoint lines $L_0, L_1, L_2$, two distinct points $P_1, P_2 \in L_0$ and two lines $M_1$ and $M_2$ containing $P_1$ and $P_2$ respectively, and not intersecting $L_1$ or $L_2$.
The set $\mathcal{S}= L_0 \cup L_1 \cup L_2 \cup M_1 \cup M_2$ of $|\mathcal{S}|=5q^2+3$ points will be an identifying code.

Since $\mathcal{S}$ contains a whole line, the set $\mathcal{S}$ is a dominating set.

A point $A$ on a line $N_1$ of $\mathcal{S}$ is clearly separated from all the points that are not on
$N_1$ since it is adjacent to all the points of $N_1$. The point $A$ is also separated from all the other points of $N_1$ since they have different projections on any line $N_2$ of $\mathcal{S}$ not intersecting $N_1$. Hence all the points of $\mathcal{S}$ are separated from all the other points.

If two points $R$ and $Q$ have the same neighbourhood on $\{L_0, L_1, L_2\}$, then this neighbourhood consists of collinear points by Lemma~\ref{lem:gqq2q3space}. If the line containing these points also contains $P_1$, then the projections of $R$ and $Q$ on the line $M_2$ are different. If the line would contain $P_2$, then the projections of $R$ and $Q$ on the line $M_1$ are different. Hence, $\mathcal{S}$ is a separating set.

To conclude the proof, note that as before we can remove a point on each line of $\mathcal{S}$ and still have an identifying code (remove a point on $L_1$, that is not a neighbour of $P_1$ or $P_2$, and remove its 4 distinct projections on the lines $L_0, L_2,M_1,M_2$).
\end{proof}

The next proposition gives a lower bound on the size of any identifying code of a $\GQ(q^2,q)$. In particular, the order of our previous construction is optimal. The proof is similar to the proof of Proposition~\ref{prop:lowT2O}.

\begin{proposition}
Let $q$ be a prime power. Any identifying code of  a $\GQ(q^2,q)$ has size at least $2q^2-2$.
\end{proposition}

\section*{Conclusion and Perspectives}

We provide identifying codes for several vertex-transitive families of graphs which have size of the same order as the fractional value. Since the graphs considered have diameter 2, our results can be extended to locating-dominating sets and to metric dimension, providing new constructions of optimal order for such sets in some families of strongly regular graphs.

Paley graphs are an example of a family of graphs for which the optimal order for the size of identifying codes is at a logarithmic factor of the fractional value. However, the fractional value is bounded by a constant. It would be interesting to exhibit a family of graphs for which the optimal values of integer and fractional identifying codes do not have the same order and such that the fractional value is not bounded by a constant.

\subsection*{Acknowledgements}
We thank Nathann Cohen for helping us with implementations in Sage.


\end{document}